\documentclass[11pt]{amsart}
\usepackage{amssymb}
\usepackage{amscd}
\usepackage{verbatim}
\usepackage[usenames,dvipsnames]{pstricks}
\usepackage{epsfig}
\usepackage{pst-grad} 
 \usepackage{pst-plot} 
\usepackage{verbatim}
\usepackage{curves}
\usepackage{geometry}
\usepackage{tikz}
\usetikzlibrary{patterns}

\usepackage{enumerate}

\usepackage{geometry}
\newtheorem{thm}{Theorem}[section]
\newtheorem{lemma}[thm]{Lemma}

\newtheorem{cor}[thm]{Corollary}

\theoremstyle{definition}

\theoremstyle{remark}

\numberwithin{equation}{section}

\newcommand{\veps}{\varepsilon}

\def \intx {\stackrel{\circ}{X}}

\def \mrn {{\mathbb R}^n}

\def \mr {{\mathbb R}}
\def \tvphi {\widetilde{\vphi}}

\def \ms {{\mathbb S}}

\def \mcs {{\mathcal S}}

\def \mch {{\mathcal H}}

\def \mcr {{\mathcal R}}

\def \eps {\varepsilon}   
\def \la {\lambda}

\def \lan {\langle}   
\def \ran {\rangle}   
\def \del {\delta}

\def \ha{ {\frac{1}{2}}}

\def \oq {\frac{1}{4}}

\def \p {\partial}

\def \rao#1 {\frac{\p}{\p #1} #1}

\numberwithin{equation}{section}

\usepackage{latexsym,eucal,amsfonts,amssymb,amsmath,graphicx}

\newcommand{\vphi}{\varphi}

\newcommand{\supp}{\textrm{Supp}}

\newcommand{\beq}{\begin{equation}}
  \newcommand{\eeq}{\end{equation}}

\def \intx {\stackrel{\circ}{X}}

\def \mrn {{\mathbb R}^n}

\def \mr {{\mathbb R}}

\def \ms {{\mathbb S}}

\def \mcs {{\mathcal S}}

\def \mch {{\mathcal H}}

\def \mcr {{\mathcal R}}

\def \intx {\overset{\circ}{X}}

\def \eps {\varepsilon}
\def \la {\lambda}

\def \novt {\frac{n}{2}}

\def \lan {\langle}
\def \ran {\rangle}
\def \del {\delta}

\input epsf
\usepackage{graphicx}
\begin{document}
\title[Local Support Theorem]
{A Local Support Theorem for the Radiation Fields on
  asymptotically Euclidean Manifolds}
\author{Ant\^onio S\'a Barreto}
\thanks{The author was partly supported by the NSF under grant DMS 0901334}
\address{Department of Mathematics, Purdue University \newline
\indent 150 North University Street, West Lafayette IN  47907, USA}
\email{sabarre@math.purdue.edu}

\begin{abstract}
We prove a local support theorem for the radiation fields on asymptotically Euclidean manifold which partly generalizes the local support theorem for the Radon transform.
\end{abstract}
\maketitle

\section{Introduction}\label{introduction}

The class of asymptotically Euclidean manifolds introduced by Melrose \cite{megs,mesc}  consists of   $C^\infty$  compact manifolds $X$ with boundary $\p X,$  equipped with a Riemannian metric that is $C^\infty$ in the interior of $X$ and singular at $\p X,$ where it has an expansion
\begin{gather} 
g= \frac{dx^2}{x^4}+ \frac{\mch}{x^2}, \label{metric0}
\end{gather}
where $x$ is a defining function of $\p X$ (that is $x\in C^\infty(X),$ $x\geq 0,$  $x^{-1}(0)=\p X,$  and $dx\not=0$ at $\p X$), and $\mch$ is a  $C^\infty$ symmetric 2-tensor such that $h_0=\mch|_{\p X}$ defines a metric on $\p X.$    The motivation for this definition comes from the fact that in polar coordinates $(r,\theta)$ the Euclidean metric  has the form $g_E=dr^2+r^2 d\omega^2,$ where $d\omega^2$ is the induced metric on $\ms^{n-1}.$ If one then uses the compactification $x=\frac{1}{r},$ for $r>C,$ the metric $g$ takes the form $$g_E=\frac{dx^4}{x^4}+\frac{d\omega^2}{x^2}, \text{ near } \{x=0\}.$$

 It was pointed out in \cite{megs} that any two boundary defining functions $x$ and $\tilde{x}$ for which \eqref{metric0} holds, must satisfy $x-\tilde{x}=O(x^2),$ and hence
 $\mch|_{\p X}$ is uniquely determined by the metric $g.$  It was shown in \cite{js3} that fixed $h_0=\mch|_{\p X},$ there exists a unique defining function $x$ near $\p X$ such that
\begin{gather} 
g= \frac{dx^2}{x^4}+ \frac{h(x)}{x^2},  \text{ in } (0,\eps)\times \p X,  \label{metric}
\end{gather}
where $h(x)$ is a $C^\infty$one-parameter family of metrics on $\p X$ and $h(0)=h_0.$

We will consider solutions to the Cauchy problem for the wave equation, 
\begin{gather}
\begin{gathered}
\left(D_t^2-\Delta_g\right)u(t,z)=0 \text{ on } (0,\infty) \times X \\
u(0,z)=f_1(z), \;\ \p_t u(0,z)=f_2(z),
\end{gathered}\label{waveeq}
\end{gather}
where $\Delta_g$ is the (positive) Laplace operator corresponding to the metric $g.$  The forward radiation field was defined by Friedlander \cite{fried0,fried1} as
\begin{gather}
\mcr_+(f_1,f_2)(s,y)= \lim_{x\rightarrow 0} x^{-\frac{n}{2}} D_t u(s+\frac{1}{x},x,y), \label{fradf}
\end{gather}
where  $n$ is  the dimension of $X.$ In the case of odd-dimensional Euclidean space, this is also known as the Lax-Phillips transform, and is given by
\begin{gather*}
\mcr_+(f_1,f_2)(s,\omega) = (2(2\pi))^{\frac{n-1}{2}} \left( D_s^{\frac{n+1}{2}} R f_1(s,-\omega)- D_s^{\frac{n-1}{2}} R f_2(s,-\omega)\right), 
\end{gather*}
 where  $R$  is the Radon transform $Rf(s,\omega)=\int_{\lan x,\omega\ran=s} f(x) \; d\mu(x),$  and  $\mu(x)$  is the Lebesgue measure on the hyperplane  $\lan x,\omega \ran =s.$  The well known theorem of Helgason \cite{helgrt} states that if $f \in \mcs(\mr^n)$ (the class of Schwartz functions), and $Rf(s,y)=0$  for $s<\leq\rho,$  then $f(z)=0$ for $|z|\geq \rho.$  One should notice that $Rf(-s,-\omega)=Rf(s,\omega),$ if $Rf(s,\omega)=0$ for $s\leq -\rho,$ then $Rf(s,\omega)=0$ for $s\geq \rho$.   Wiegerinck \cite{wieg} proved local versions of this result. More precisely, he proved that if $f\in C_0^\infty(\mr^n),$ then $f(z)=0$ on the set 
 $$\{z\in \mr^n:  \lan z, \omega\ran=s, \text{ and } (s,\omega)\not\in \supp (Rf).\}.$$

 Wiegerinck's  proof  relies very strongly on analyticity  properties  of the Fourier transform of  functions in $C_0^\infty(\mr^n),$ and the fact that the Fourier transform in the $s$ variables of $Rf(s,\omega)$ satisfies $\widehat{Rf}(\la,\omega)=\widehat{f}(\la\omega),$ where the right hand side essentially is the Fourier transform of $f$ in polar coordinates.  Such a result  is not likely to hold in more general situations.  Here we will prove the following
\begin{thm}\label{main} Let  $(X,g)$ be an asymptotically Euclidean manifold,  let $\Omega \subset \p X$ be an open subset, and let $f\in C_0^\infty(\intx).$ 
  Let $\veps>0$ be such that \eqref{metric} holds on $(0,\veps)\times \p X$ and let $\bar{\veps}=\min\{\veps,-\frac{1}{s_0}\}.$ 
Then  $\mcr_+(0,f)(s,y)=0$ for $s\leq s_0<0$ and $y \in \Omega,$  if and only if for every $(x,y),$ $x\in(0,\bar{\eps}),$ and $y\in \Omega,$
\begin{gather}
 d_g((x,y), \supp f) \geq s_0+ \frac{1}{x}. \label{distsup}
\end{gather}
\end{thm}
 In the case where $\Omega=\p X,$ this result was proved in \cite{sberf}. In the case of radial solutions of semilinear wave equations $\Box u=f(u)$ in $\mr\times \mr^3,$ with critical non-linearities, and $\Omega=\ms^{n-1}$  a similar result was proved in \cite{basa}. In the case of asymptotically hyperbolic manifolds results of this nature  have been proved in \cite{guilsa,hosa,sbhrf}.
 
 In Euclidean space, the polar distance $r=\frac{1}{x},$ and hence \eqref{distsup} implies that if $z\in \supp(f),$ then for every $p,$ such that  $p=r\omega,$ $\omega \in \Omega,$ and   $|p|>|s_0|,$
 \begin{gather*}
 |z-p|\geq |p|-|s_0|.
 \end{gather*}
 In particular this implies that if 
 \begin{gather*}
 |z|^2-2r\lan z, \omega\ran \geq |s_0|^2-2r |s_0|, \;\  r>|s_0|,  \;\ \omega \in \Omega.
 \end{gather*}
If we let $r\rightarrow \infty,$ it follows that if $z\in \supp(f)$ then $\lan z, \omega \ran \leq |s_0|.$ See Fig. \ref{fig4} 

\begin{figure}
\scalebox{.7} 
{
\begin{pspicture}(0,-1.23)(13.16,1.89)
\psline[linewidth=0.04cm](7.24,1.29)(11.72,-1.21)
\psline[linewidth=0.04cm](5.4924355,1.2975421)(1.3675644,-1.1575421)
\usefont{T1}{ptm}{m}{n}
\rput(6.31,1.695){$\Omega$}
\usefont{T1}{ptm}{m}{n}
\rput(10.94,0.675){$\lan z, \omega\ran\geq |s_0|$}
\usefont{T1}{ptm}{m}{n}
\rput(2.18,0.715){$\lan z, \omega\ran\geq |s_0|$}
\psarc[linewidth=0.04](6.35,-0.16){1.69}{57.36249}{121.551384}
\usefont{T1}{ptm}{m}{n}
\rput(6.61,0.755){$|s_0|$}
\usefont{T1}{ptm}{m}{n}
\rput(8.53,1.195){$f=0$}
\usefont{T1}{ptm}{m}{n}
\rput(4.37,1.235){$f$=0}
\usefont{T1}{ptm}{m}{n}
\rput(11.26,0.275){$\omega\in \Omega$}
\usefont{T1}{ptm}{m}{n}
\rput(1.92,0.315){$\omega\in \Omega$}
\rput{-180.05002}(12.640708,1.0968579){\psarc[linewidth=0.04,linestyle=dashed,dash=0.16cm 0.16cm](6.3201146,0.5511877){1.0810385}{315.1605}{236.88087}}
\psline[linewidth=0.04,linestyle=dashed,dash=0.16cm 0.16cm](6.98,1.37)(6.34,0.37)(5.58,1.33)(5.58,1.31)
\end{pspicture} 
}
\caption{If $f\in C_0^\infty(\mr^n)$ and  $\mcr(0,f)(s,\omega)=0$ for $s\leq s_0<0$ and $\omega \in \Omega\subset \ms^{n-1},$ then $f(z)=0$ if $\lan z, \omega \ran \geq |s_0|$ for all $\omega\in \Omega.$}
\label{fig4}
\end{figure}
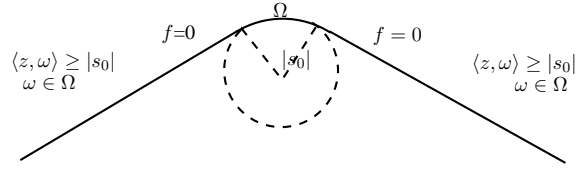
 This result can be rephrased in terms of the sojourn times for geodesics in $\mr^n.$    Let $\omega \in \ms^{n-1}$ and $\gamma_{z,\omega}(t)=z+t \omega$ be  a geodesic starting at a point $z\in \mr^n$ in the direction of the unit vector $\omega.$  The sojourn time along $\gamma_{z,\omega}$ is defined to be 
$$S(z,\omega)=\lim_{t\rightarrow \infty} (t-|\gamma_{z,\omega}(t)|),$$ see for example \cite{sawu}.  But
\begin{gather*}
t-|\gamma(t)|=t-\left(|z|^2+t^2+2t \lan z, \omega\ran \right)^\ha= t-t\left( 1+\frac{2}{t} \lan z, \omega \ran + \frac{|z|^2}{t^2}\right)^\ha= -\lan z,\omega\ran + O(t^{-1}).
\end{gather*}
So Theorem \ref{main} says that  if $z\in \supp(f)$ and $\omega\in \Omega,$ then $S(z,\omega)\geq s_0,$ (i.e $\lan z, \omega\ran \leq |s_0|.$).

The connection between sojourn times and scattering theory is well known, see for example \cite{guil}. The sojourn times 
on asymptotically Euclidean manifolds was studied in \cite{sawu}.  If $(X,g)$ is an asymptotically Euclidean manifold,  $z\in \intx$   and $\gamma(t)$ is a geodesic  parametrized by the arc-length such that $\gamma(0)=z$ and $\lim_{t\rightarrow \infty}\gamma(t)=y\in \p X,$ the sojourn time along $\gamma$ is defined by
\begin{gather*}
S(z,\gamma)=\lim_{t\rightarrow \infty} (t-\frac{1}{x(\gamma(t))}),
\end{gather*}
where $x$ is a boundary defining function as in \eqref{metric}.  We obtain the following result from Theorem \ref{main}:
\begin{cor}\label{conse} Let $f\in C_0^\infty(\intx)$ and let $\Omega\subset \p X$ be an open subset. Suppose that $\mcr(0,f)(s,y)=0$ for every $s\leq  s_0<0$ and $y \in \Omega.$  If $z\in \intx$ is such that there exists $y \in \Omega$ and a geodesic $\gamma$ parametrized by the arc-length such that $\gamma(0)=z,$  $\lim_{t\rightarrow \infty}\gamma(t)=y,$ and $S(z,\gamma)< s_0,$ then $f(z)=0.$
\end{cor}
\begin{proof} If $z$ and $\gamma(t)$ are as in the hypothesis, then since $t$ is the arc-length parameter $ d(z,\gamma(t))\leq t.$  If $S(z,\gamma)<s_0,$ then there exists $T>0$ such that $ t-\frac{1}{x(\gamma(t))}<s_0$ for $t>T.$  If $T$ is large enough $\gamma(t)\in (0,\eps) \times \Omega,$ and for $t>T,$
\begin{gather*}
d(z,\gamma(t))\leq t < s_0+\frac{1}{x(\gamma(t))},
\end{gather*}
thus $z\not\in \supp(f),$ and hence $f(z)=0.$
\end{proof}

\section{The proof of Theorem \ref{main}}

 Suppose that $f\in C_0^\infty(\intx)$ and \eqref{distsup} holds for $x\in(0,\eps)$ and $y \in \Omega.$  Let $u$ be the solution of \eqref{waveeq} with initial data $(0,f),$ and let 
$v(x,s,y)=x^{-\frac{n}{2}} u(s+\frac{1}{x},x,y).$    By finite speed of propagation,
\begin{gather*}
u(t,(x,y))=0 \text{ if } t\leq d_g((x,y),\supp(f)).
\end{gather*}
This implies that
\begin{gather*}
v(x,s,y)=0 \text{ if } s\leq d_g((x,y),\supp(f))-\frac{1}{x}.
\end{gather*}
If $d_g((x,y),\supp(f))-\frac{1}{x}\geq s_0,$ then $v(x,s,y)=0$ if $x\in(0,\eps),$ $y\in \Omega$ and $s\leq s_0.$ In particular, $\mcr(0,f)(s,y)=0$ if $s\leq s_0$ and $y \in \Omega.$
The converse is much harder to prove. 

Since $f\in C_0^\infty(\intx),$ there exists $x_0\in (0,\eps)$ such that 
$\supp(f)\ \subset \{ x\geq x_0\}.$   If  $-\frac{1}{x_0}<s_0,$ the result is obvious. Indeed,  if $x<x_0,$ then $d((x,y), \supp(f))> d((x,y), (x_0,y))=\frac{1}{x}-\frac{1}{x_0}>  \frac{1}{x}+s_0,$ then, by the definition of support,  $f(z)=0$ if there exists $(x,y)$ such that $d(z, (x,y))\leq s_0+\frac{1}{x}.$ 
So we will assume from now on that $\supp (f) \subset \{x\geq x_0\},$ and $-\frac{1}{x_0}<s_0,$ see Fig. \ref{fig0}.   
\begin{figure}
\scalebox{.6} 
{
\begin{pspicture}(0,-4.87)(10.378055,4.89)
\definecolor{color264g}{rgb}{0.8,0.8,0.8}
\psline[linewidth=0.04cm](1.6780548,4.03)(1.6980548,-4.85)
\psline[linewidth=0.04cm](1.6780548,1.09)(9.678055,1.13)
\psbezier[linewidth=0.04](2.2380548,-4.17)(2.2180548,-3.05)(3.9588823,-1.4511017)(4.8380547,-0.87)(5.7172275,-0.2888983)(8.298055,0.69)(8.958055,0.53)
\usefont{T1}{ptm}{m}{n}
\rput(10.068055,1.035){$x$}
\usefont{T1}{ptm}{m}{n}
\rput(2.1280549,4.695){$s$}
\usefont{T1}{ptm}{m}{n}
\rput(1.2680548,0.555){$s_0$}
\usefont{T1}{ptm}{m}{n}
\rput(3.3080547,1.335){$x_0$}
\psline[linewidth=0.04cm,linestyle=dashed,dash=0.16cm 0.16cm](1.6780548,0.31)(7.678055,0.33)
\psline[linewidth=0.04cm,linestyle=dashed,dash=0.16cm 0.16cm](3.0180547,1.09)(3.0380547,-2.51)
\psline[linewidth=0.04cm,linestyle=dashed,dash=0.16cm 0.16cm](1.6580548,-2.47)(3.0580547,-2.47)
\usefont{T1}{ptm}{m}{n}
\rput(3.2580547,-3.285){$v=0$}
\usefont{T1}{ptm}{m}{n}
\rput(4.728055,-3.685){by finite speed of propagation}
\usefont{T1}{ptm}{m}{n}
\rput{88.99379}(-0.98454493,-1.2780112){\rput(0.1580548,-1.145){$v=0$}}
\psline[linewidth=0.04,fillstyle=gradient,gradlines=2000,gradbegin=color264g,gradend=color264g,gradmidpoint=1.0](0.6380548,0.31)(0.65805477,-4.11)(1.7180548,-4.11)(1.6980548,0.33)(0.6380548,0.33)(0.65805477,0.31)
\psline[linewidth=0.04,fillstyle=gradient,gradlines=2000,gradbegin=color264g,gradend=color264g,gradmidpoint=1.0](1.7380548,-2.49)(3.0580547,-2.47)(2.6180549,-3.07)(2.318055,-3.67)(2.2380548,-4.11)(1.6980548,-4.11)
\end{pspicture} 
}
\caption{$v(x,s,y)=0$ for $-\frac{1}{x}<s\leq s_0$ by finite speed of propagation, and  for $y\in \Omega,$ $v=0$ for $x\leq 0$ and $s\leq s_0$ because $\mcr(0,f)=0$ for $s\leq s_0$.}
\label{fig0}
\end{figure}
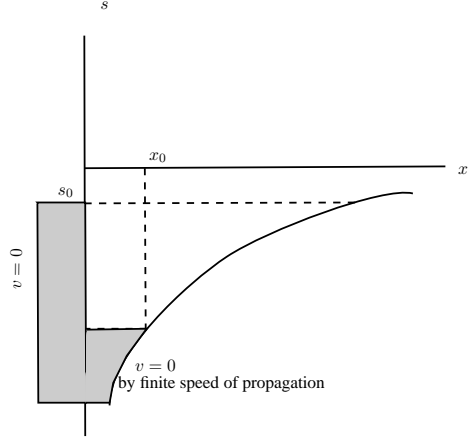

The one-parameter family of metrics $h(x),$ $x\in [0,\eps],$  has a (non-unique)  $C^\infty$ extension to  $[-\eps,\eps],$ and Friedlander proved that , fixed the etension of $h,$
$v(x,s,y)=x^{-\frac{n}{2}} u(s+\frac{1}{x},x,y)$ has a unique extension to  $C^\infty([-\eps,\eps] \times \mr \times \p X)$ which satisfies
\begin{gather}
\begin{gathered}
Pv=0 \text{ for } s>-\frac{1}{x} \\
v(x,-\frac{1}{x},y)=0, \;\ \p_s v(x, -\frac{1}{x},y)= x^{-\novt} f(x,y),
\end{gathered}\label{waveeq1}
\end{gather}
where $P$ is the wave operator written in coordinates $(x,s,y),$ with $s=t-\frac{1}{x},$ which is
\begin{gather*}
P=D_x(2D_s+x^2 D_x) + \Delta_{h} +iA(D_s+x^2D_x) +B, \\
A=\frac{1}{\sqrt{|h|}} \p_x \sqrt{|h|} , \;\  B= \frac{n-1}{2}(\frac{3-n}{2}+xA),
\end{gather*}
$|h|$ is the volume element of the metric $h$ and $\Delta_{h}$ is the (positive) Laplacian with respect to $h.$    By finite speed of propagation, $v=0$ if 
$s\leq -\frac{1}{x_0},$ and the formal power series argument carried out in section 4 of \cite{sberf}  shows that  $\p_x^k v(0,s,y)=0$ for $k=0,1,2,...,$ provided $s< s_0$ and  $y\in \Omega.$ 
 This implies that 
\begin{gather}
\begin{gathered}
v(x,s,y)=0 \text{ if  }  x<0, \;\ s<s_0 ,\;\ y \in \Omega, \\
v(x,s,y)=0 \text{ if }   s\leq -\frac{1}{x_0}, \;\ s>-\frac{1}{x}, \;\ 0<x<\eps,
\end{gathered}\label{vanv}
\end{gather}
see Fig. \ref{fig0}.    We begin by proving 
\begin{lemma}\label{lemma1}  Let $v(x,s,y)$ satisfy \eqref{waveeq1} and \eqref{vanv} for $x_0\in (0,\eps)$  and $-\frac{1}{x_0}<s_0<0.$  Let $y_0\in \Omega$ and suppose that $\{y:|y-y_0|<r\}\subset \Omega.$  Let $N$ be such that $s_0+\frac{1}{x_0}<\frac{N}{4}.$ There exists $\del>0,$  depending on $r$ and on derivatives up to order two of the tensor  $h(x),$  $x\in [-\eps,\eps],$  such that $v=0$ on the set
\begin{gather}
\left\{x<\frac{\del}{3N}(s_0+\frac{1}{x_0}) , \; |y-y_0|<\left(\frac{\del^\ha}{3N}(s_0+\frac{1}{x_0}) \right)^\ha, \;\ -\frac{1}{x}<s< -\frac{1}{x_0}+\frac{1}{3N}(s_0+\frac{1}{x_0})\right\}. \label{vanv1}
\end{gather} 
\end{lemma}
\begin{proof}
We should point out that the fact that the bound on $|s+\frac{1}{x_0}|$ does not depend on $\del$ or $r,$ is due to the fact that the coefficients of the operator $P$ do not depend on $s.$

 Let $(\xi,\sigma,\eta)$ denote dual local coordinates to $(x,s,y).$ The principal symbol of $P$ is
\begin{gather}
p=2\xi\sigma + x^2 \xi^2+h, \label{symbol}
\end{gather}
and the Hamilton vector field of $p$ is equal to
\begin{gather*}
H_p= 2(\sigma+x^2\xi)\p_x + 2\xi \p_s -(2x\xi^2+\p_x h) \p_\xi+ \sum_{j=1}^n (\p_{\eta_j} h \p_{y_j}- \p_{y_j} h \p_{\eta_j}).
\end{gather*}
 Suppose that $y_0=0\in \Omega$ and let $y$ be local coordinates valid in $\{|y|<r\}\subset \Omega.$  Let
\begin{gather*}
\vphi= -2x-2\del(s-a) -x(s-a) - \del^\ha |y|^2 , \text{ where } a=-\frac{1}{x_0}, \text{ and } \\
\widetilde{\vphi}= -x-\del(s-a)-x(s-a).
\end{gather*}
Then 
\begin{gather}
\begin{gathered}
v=0  \text{ on the set } \{ \widetilde{\vphi}>0\} \cap\left( \{x\leq 0, s\leq s_0, \; |y|<r \}  \cup \{ -\frac{1}{x} <s<-\frac{1}{x_0}, \;\ 0<x< x_0, \; |y|<r \} \right), \\
\end{gathered}\label{vanishtil}
\end{gather}
see Fig. \ref{fig01}.
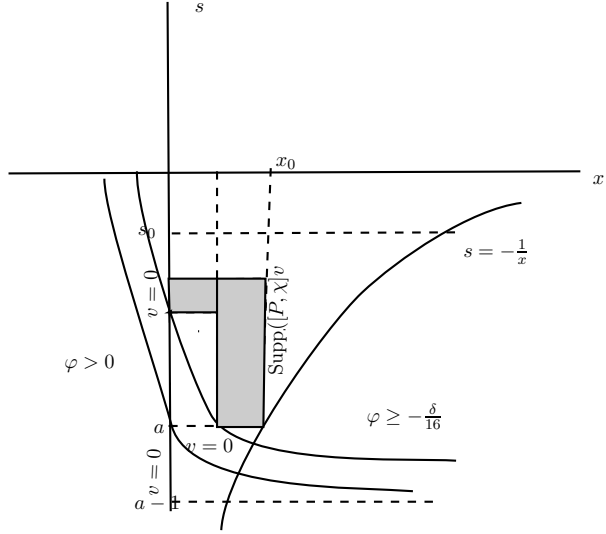
\begin{figure}
\scalebox{.7} 
{
\begin{pspicture}(0,-5.06)(11.52,5.08)
\definecolor{color512g}{rgb}{0.8,0.8,0.8}
\psline[linewidth=0.04cm](3.08,-4.68)(3.02,5.0)
\psline[linewidth=0.04cm](0.0,1.74)(10.86,1.78)
\psbezier[linewidth=0.04](4.04,-5.04)(4.106109,-3.8307755)(6.02,-1.18)(6.7089453,-0.52889794)(7.397891,0.12220408)(8.774181,1.0608163)(9.74,1.18)
\usefont{T1}{ptm}{m}{n}
\rput(11.21,1.625){$x$}
\usefont{T1}{ptm}{m}{n}
\rput(3.63,4.885){$s$}
\usefont{T1}{ptm}{m}{n}
\rput(2.63,0.625){$s_0$}
\usefont{T1}{ptm}{m}{n}
\rput(2.86,-3.115){$a$}
\usefont{T1}{ptm}{m}{n}
\rput(9.26,0.245){$s=-\frac{1}{x}$}
\psline[linewidth=0.04cm,linestyle=dashed,dash=0.16cm 0.16cm](5.24,-1.24)(5.24,-1.22)
\usefont{T1}{ptm}{m}{n}
\rput(5.27,1.945){$x_0$}
\psline[linewidth=0.04cm,linestyle=dashed,dash=0.16cm 0.16cm](3.08,-2.26)(3.06,-2.26)
\psline[linewidth=0.04cm,linestyle=dashed,dash=0.16cm 0.16cm](3.1,0.6)(8.48,0.62)
\usefont{T1}{ptm}{m}{n}
\rput(1.54,-1.875){$\vphi>0$}
\psline[linewidth=0.04cm,linestyle=dashed,dash=0.16cm 0.16cm](3.04,-3.06)(4.24,-3.06)
\psline[linewidth=0.04cm,linestyle=dashed,dash=0.16cm 0.16cm](3.6,-1.26)(3.62,-1.26)
\psline[linewidth=0.04cm,linestyle=dashed,dash=0.16cm 0.16cm](3.96,-3.08)(3.96,1.78)
\usefont{T1}{ptm}{m}{n}
\rput{86.89689}(-1.332562,-6.45324){\rput(2.74,-3.935){$v=0$}}
\psline[linewidth=0.04cm,linestyle=dashed,dash=0.16cm 0.16cm](4.24,-3.06)(4.8,-3.08)
\psline[linewidth=0.04cm,linestyle=dashed,dash=0.16cm 0.16cm](4.8,-3.1)(4.98,1.84)
\usefont{T1}{ptm}{m}{n}
\rput{89.40736}(4.026995,-6.168865){\rput(5.13,-1.035){$\supp([P,\chi] v$}}
\usefont{T1}{ptm}{m}{n}
\rput(3.82,-3.435){$v=0$}
\usefont{T1}{ptm}{m}{n}
\rput{89.34301}(2.119077,-3.243516){\rput(2.7,-0.555){$v=0$}}
\psline[linewidth=0.04cm,linestyle=dashed,dash=0.16cm 0.16cm](3.1,-4.5)(8.06,-4.5)
\psbezier[linewidth=0.04](1.82,1.64)(1.82,0.92)(2.7532682,-1.7980325)(3.1,-3.06)(3.4467318,-4.3219676)(6.72,-4.22)(7.68,-4.3)
\usefont{T1}{ptm}{m}{n}
\rput(2.83,-4.535){$a-1$}
\psbezier[linewidth=0.04](2.44,1.78)(2.4,1.04)(3.2932682,-1.8180325)(3.84,-2.84)(4.3867316,-3.8619676)(7.56,-3.66)(8.5,-3.72)
\psdots[dotsize=0.12](3.1,-0.86)
\usefont{T1}{ptm}{m}{n}
\rput(7.53,-2.875){$\vphi\geq -\frac{\del}{16}$}
\psline[linewidth=0.04,linestyle=dashed,dash=0.16cm 0.16cm,fillstyle=gradient,gradlines=2000,gradbegin=color512g,gradend=color512g,gradmidpoint=1.0](3.04,-0.28)(4.88,-0.26)(4.88,-0.9)(3.08,-0.9)(3.06,-0.92)
\psline[linewidth=0.04,fillstyle=gradient,gradlines=2000,gradbegin=color512g,gradend=color512g,gradmidpoint=1.0](3.96,-0.28)(3.96,-3.08)(4.84,-3.08)(4.88,-0.26)(3.04,-0.26)(3.06,-0.96)(3.08,-0.88)(3.04,-0.9)(3.94,-0.9)(3.96,-0.92)
\end{pspicture} 
}
\caption{$v(x,s,y)=0$ for $\tvphi>0$ in a neighborhood of $x=0,$ $s=a$ and $y=0.$}
\label{fig01}
\end{figure}

We also have
\begin{gather}
\begin{gathered}
p(x,s,y,d\vphi)= 2(2\del+x)(2+s-a)+x^2(2+s-a)^2+h(x,y,d_y \vphi)> 2\del \\
\text{ provided } |s-a|< 1, \;\ |x|< \del.
\end{gathered}\label{nonchar}
\end{gather}
and 
\begin{gather}
\begin{gathered}
H_p \vphi =-2(\sigma+x^2\xi)(2+s-a)-2\xi(2\del+x) -\del^\ha H_h|y|^2, \\
H_p^2 \vphi=-8\xi(\sigma+x^2\xi)(1+x(2+s-a))+4x(2\del+x+x^2(2+s-a))\xi^2 \\ -2\del^\ha(\sigma+x^2\xi)\p_x H_h |y|^2-
2(2\del+x+x^2(2+s-a))\p_x h-\del^\ha H_h^2 |y|^2.
\end{gathered}\label{hpfi}
\end{gather}
If $p=H_p\vphi=0,$ then
\begin{gather*}
h=\left(\frac{2(2\del+x)}{2+s-a}+x^2\right)\xi^2+\frac{\del^\ha}{2+s-a} \xi H_h |y|^2,
\end{gather*}
and hence
\begin{gather*}
h+\frac{1}{2+s-a} (H_h |y|^2)^2 \geq \left(\frac{3\del+2x}{2+s-a}+ x^2\right) \xi^2.
\end{gather*}
If $|s-a|<1,$ $|x|<\del$ and $C>0$  is such that
\begin{gather*}
(H_h |y|^2)^2\leq C h,
\end{gather*}
then
\begin{gather}
h\geq C \del \xi^2. \label{boundh}
\end{gather}
Here, and from now on, $C>0$ denotes a constant which depends on the metric $h(x),$ $x\in [-\eps,\eps].$ 

If $p=0,$ then $2(\sigma+x^2\xi)\xi=-h +x^2\xi^2,$ and we deduce from \eqref{hpfi} that
\begin{gather*}
H_p^2\vphi=4(1+x(2+s-a))h +\frac{2\del}{2+s-a} (H_h |y|^2)(\p_x H_h |y|^2)-\del^\ha H_h^2 |y|^2\\
-2(2\del+x+x^2(2+s-a)) \p_x h + \frac{2(2\del+x)}{2+s-a} \xi \p_x H_h |y|^2 + 8\del x \xi^2,
\end{gather*}
and if $\del<\frac{1}{10}$ 
\begin{gather*}
H_p^2\geq 3h- 100 \del^\ha( (\p_x H_h|y|^2)^2+ (H_h|y|^2)^2+ |\p_x h|+ |H_h^2|y|^2|)-20 \del^2 \xi^2.
\end{gather*}
We can pick $\del_0$ such that 
\begin{gather}
3h- 100 \del^\ha( (\p_x H_h|y|^2)^2+ (H_h|y|^2)^2+ |\p_x h|+ |H_h^2|y|^2|)>h, \text{ if }  \del<\del_0 \label{boundh0}
\end{gather}
we can use \eqref{boundh} to conclude that if $|x|<\del$ and $\del<\del_0,$
\begin{gather}
H_p^2 \vphi\geq h-20\del^2\xi^2=\ha h  + C\del \xi^2. \label{boundh2}
\end{gather}

Hence we conclude that if $|x|<\del<\del_0$ and  $|s-a|<1,$ then
\begin{gather*}
p(d\vphi)> \del  \text{ and  if }  p=H_p \vphi=0  \Rightarrow H_p^2 \vphi>0  \text{ provided } (\xi,\sigma,\eta)\not=0.
 \end{gather*}
So the level surfaces of $\vphi$ are strongly pseudoconvex with respect to $P$ in the region 
\begin{gather}
U=\{ |x|< \del, \; |s-a|<\bar{s} , \;  |y|< r   \}, \;\ \bar{s}=\min\{1, s_0-a\} \label{regionu}
\end{gather}
and therefore it follows from Theorem 28.2.3 and Proposition 28.3.3 of \cite{hormander4}  that if $Y\subset\subset U$ and $\la>0$ and $K>0$ are large enough, then for
$\psi=e^{\la \vphi},$
\begin{gather}
\begin{gathered}
\sum_{|\alpha|<2} \tau^{2(2-|\alpha|)-1} \int |D^\alpha w|^2 e^{2\tau \psi} dxdsdy \leq \\
K(1+C\tau^{-\ha}) \int |Pw|^2 e^{2\tau \psi} dxdsdy, \;\ w\in C_0^\infty(Y), \;\ \tau>1.
\end{gathered}\label{carle}
\end{gather}
Let 
\begin{gather*}
U_\gamma= \{ |x|<\gamma \del, \; |s-a|< \gamma\bar{s} , \;  |y|< \gamma r   \},
\end{gather*}
and $\chi(x,s,y) \in C_0^\infty$ be such that  $\chi=1$ on $U_\oq$ and $\chi=0$ outside $U_\ha.$

  Therefore
  \begin{gather}
  \supp[P,\chi] \subset  \overline{U_\ha\setminus U_\oq}.\label{supcom}
  \end{gather}
  On the other hand, $v=0$ if $\tvphi>0,$ and $|x|<\del,$ $|s-a|<\bar{s},$ and $|y|<r,$ and
  \begin{gather*}
  \vphi= \tvphi-(x+\del(s-a)+\del^\ha |y|^2),
  \end{gather*}
  so we conclude that
  \begin{gather*}
  \vphi\leq -(x+\del(s-a)+\del^\ha |y|^2) \text{ on the support of } v.
  \end{gather*}
  So we deduce from \eqref{supcom} that,  provided that $\del^\ha <\frac{r^2}{4},$ and  $N$ is such that $\frac{s_0-a}{N}<\frac{1}{4},$
  \begin{gather*}
  \vphi \leq - \min_{\overline{V_\ha\setminus V_{\frac{1}{4}}}} (x+\del(s-a)+\del^\ha |y|^2)=-\frac{\del(s_0-a)}{N}.
   \end{gather*}
  
So we conclude that
  \begin{gather*}
  \supp [P,\chi] v \subset \{ \vphi \leq -\frac{\del(s_0-a)}{N}\},
  \end{gather*}
 and hence we deduce from \eqref{carle} applied to $w=\chi v,$  and the fact that $P\chi v= \chi P v+[P,\chi ] v=[P,\chi] v$ that
\begin{gather*}
\sum_{|\alpha|<2} \tau^{2(2-|\alpha|)-1} \int |D^\alpha \chi v|^2 e^{2\tau \psi} dxdsdy \leq C e^{2\tau e^{-\la \frac{\del(s_0-a)}{N}}}
\end{gather*}
and we conclude that $\chi v=0$ if $\vphi\geq -\frac{\del(s_0-a)}{N}.$  Notice that $\vphi\geq -\frac{\del(s_0-a)}{3N}$  corresponds to the set
\begin{gather*}
x+\del(s-a)+\del^\ha |y|^2<\frac{\del(s_0-a)}{N}
\end{gather*}
and since $v=0$ in $\{x<0, \;\ s<s_0\} \cup \{-\frac{1}{x}<s< a\},$ we deduce that $v=0$ on the set
\begin{gather*}
\{|x|< \frac{ \del(s_0-a)}{3N}, \;\ |s-a|\leq \frac{s_0-a}{3N}, \; |y|^2\leq \frac{\del^\ha(s_0-a)}{3N}\}.
 \end{gather*}
\end{proof}
The next step on the proof is the following lemma:
\begin{lemma}\label{lemma2} Let $\Omega \subset \p X$ be an open subset and  let $u(t,z)$ be a solution to \eqref{waveeq} with initial data $(0,f).$   Suppose that \eqref{metric} holds in $(0,\eps),$ and let 
$v(x,s,y)=x^{-\novt} u(s+\frac{1}{x},x,y).$  Suppose that $v\in C^\infty([0,\eps)\times \mr \times \p X)$ and  that  $v=0$ on the set
\begin{gather*}
\{-\frac{1}{x}< s<a,  x\leq x_0=-\frac{1}{a}\;\ y\in \Omega\} \cup \{x<0, \;\ s<s_0, \;\ y\in \Omega \}.
\end{gather*}
Then $v=0$ on the set $\{-\frac{1}{x}<s<s_0, \;\ x<\min\{\eps,-\frac{1}{s_0}\}, \;\ y\in \Omega\}.$ See Fig.\ref{fig2}
\end{lemma}
\begin{proof} The main ingredients of the proof of this result are Lemma \ref{lemma1} and the following result of Tataru \cite{tataru,tataru1}:  If $u(t,z)$ is a $C^\infty$ function that satisfies
\begin{gather*}
\begin{gathered}
(D_t^2-\Delta_g +L(z,D_z))u=0 \text{ in } (\widetilde{T}, \widetilde{T}) \times \Omega, \\
u(t,z)=0 \text{ in a neighborhood of } \{z_0\} \times (-T,T), \;\ T<\widetilde{T},
\end{gathered}
\end{gather*}
where $\Omega\subset \mrn,$ $g$ is a $C^\infty$ Riemannian metric and $L$ is a first order $C^\infty$ operator (that does not depend on $t$), then
\begin{gather}
u(t,z)=0 \text{ if } |t|+ d_g(z,z_0)<T, \label{tatres}
\end{gather}
 where $d_g$ is the distance measured with respect to the metric $g.$    
 
 Let
\begin{gather}
a_0=a, \;\  a_1= a_0+ \frac{1}{3N}(s_0-a_0) \text{ and } a_j=a_{j-1} + \frac{1}{3N}(s_0-a_{j-1}).\label{sequence}
\end{gather}
 
 We know from Lemma \ref{lemma1} that for each $y_0\in \Omega,$ there exists $\del>0$ such that
 $v(x,s,y)=0$ if $x<C\del,$ $|y-y_0|<C\del$ and $s<a_1=a+\frac{s_0-a}{3N}.$    
  In particular for any $\alpha \in (0,C\del ),$ $v(\alpha,s,y)=0$ in a neighborhood of  the segment
 \begin{gather*}
 x=\alpha, \; -\frac{1}{\alpha} < s< a_1, \; y \in \{|y-y_0|<C\del\}.
 \end{gather*}
 Since $t=s+\frac{1}{\alpha},$ this implies that $u(t,x,y)= x^{\frac{n}{2}} v(x,t-\frac{1}{x},y)=0$ in a neighborhood of the segment
 \begin{gather*}
 x=\alpha, \; 0< t <a_1+\frac{1}{\alpha}, \;\ y \text{ such that } |y-y_0|< C\del.
 \end{gather*}
 But since the initial data is of the form $(0,f),$ $u(t,z)=-u(-t,z),$ and hence $u(t,z)=0$ in a neighborhood of 
 \begin{gather*}
x=\alpha, \;\  -a_1 -\frac{1}{\alpha}< t< \frac{1}{\alpha} +a_1, \;\ y \text{ such that } |y-y_0|< C\del .
 \end{gather*}
 From \eqref{tatres} we obtain
 \begin{gather*}
 u(t,z)=0 \text{ if }   |t|+ d_g(z, (\alpha,y) )<a_1 + \frac{1}{\alpha}. \label{tatres1}
 \end{gather*}
 If one picks  $z=(x,y),$ with $\eps> x>\alpha,$ then $d_g(z,(\alpha,y))= \frac{1}{\alpha}-\frac{1}{x},$ and hence in particular,
  \begin{gather*}
 u(t,x,y)=0 \text{ if }  0< t<\frac{1}{x}+a_1, \;\ x<\min\{-\frac{1}{a_1},\eps\},  \;\  |y-y_0|< C\del . \label{tatres2}
 \end{gather*}
 Since $y_0$ is arbitrary, this implies that 
\begin{gather*}
v(x,s,y)=0 \text{ on the set } \{-\frac{1}{x}<s<a_1,  \;\ x<\min\{-\frac{1}{a_1}, \eps\} , \;\ y\in \Omega\}.
\end{gather*}

Applying this argument above $j$ times we obtain
 \begin{gather*}
 v(x,s,y)=0 \text{ on the set }   \{x\leq 0, s\leq s_0, y \in \Omega\} \cup \{ x\leq \min\{\eps,-\frac{1}{a_j}\} , \;\ -\frac{1}{x} < s \leq a_1, \; y \in \Omega\}.
 \end{gather*} 

The sequence \eqref{sequence} is increasing and $a_j< s_0.$  Let 
$L=\lim_{j\rightarrow\infty} a_j.$ Then from the definition of $a_j$ it follows that $\frac{1}{3N}(s_0-L)=0,$ and so $L=s-0.$  Since $v\in C^\infty$ it follows that

 \begin{gather*}
 v(x,s,y)=0 \text{ on the set }   \{x\leq 0, s\leq s_0, y \in \Omega\} \cup \{ x\leq \min\{\eps,-\frac{1}{s_0}\}  \;\ -\frac{1}{x} < s\leq s_0, \; y \in \Omega\}.
 \end{gather*} 

This proves the Lemma.
\end{proof}
\begin{figure}
\scalebox{.6} 
{
\begin{pspicture}(0,-4.58)(10.98,4.6)
\definecolor{color751g}{rgb}{0.8,0.8,0.8}
\psline[linewidth=0.04cm](0.96,4.44)(0.96,-4.46)
\psline[linewidth=0.04cm](0.98,0.8)(10.54,0.84)
\usefont{T1}{ptm}{m}{n}
\rput(1.48,-4.055){$v=0$}
\usefont{T1}{ptm}{m}{n}
\rput(10.67,0.645){$x$}
\usefont{T1}{ptm}{m}{n}
\rput(1.39,4.405){$s$}
\usefont{T1}{ptm}{m}{n}
\rput(0.6,-3.635){$a$}
\psline[linewidth=0.04cm,linestyle=dashed,dash=0.16cm 0.16cm](2.78,-2.4)(2.8,0.86)
\usefont{T1}{ptm}{m}{n}
\rput(2.87,1.045){$\eps$}
\usefont{T1}{ptm}{m}{n}
\rput(0.43,0.0050){$s_0$}
\psbezier[linewidth=0.04](1.36,-4.56)(1.74,-3.58)(2.976013,-1.9763635)(4.34,-1.1)(5.703987,-0.22363645)(8.84,0.32)(9.66,0.2)
\usefont{T1}{ptm}{m}{n}
\rput{90.0196}(0.83104706,-5.290763){\rput(3.06,-2.215){$v=0$ by unique continuation}}
\usefont{T1}{ptm}{m}{n}
\rput{87.40695}(-1.2388291,-2.4837968){\rput(0.68,-1.895){$v=0$}}
\psline[linewidth=0.04,fillstyle=gradient,gradlines=2000,gradbegin=color751g,gradend=color751g,gradmidpoint=1.0](0.98,0.02)(2.8,-0.02)(2.8,-2.42)(1.86,-3.62)(0.98,-3.6)
\end{pspicture} 
}
\caption{The  second step of the unique continuation across the wedge $\{-\frac{1}{x}<s< -\frac{1}{x_0},y\in \Omega\} \cup \{x< 0,s< s_0, y\in \Omega\}$ }
\label{fig2}
\end{figure}
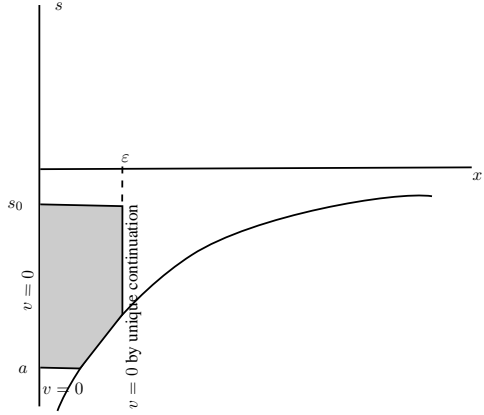

Now we can finish the proof of  Theorem \ref{main}.    Suppose $\supp(f) \subset \{x>x_0\}$ and that $\mcr_+(0,f)(s,y)=0,$ if $s<s_0$ and $y\in \Omega.$   Then  $v$ extends as a solution to \eqref{waveeq1} satisfying
\eqref{vanv}.  Then Lemma \ref{lemma1} and Lemma \ref{lemma2} imply that 
\begin{gather*}
v=0 \text{ in the set } \{x<\min\{\eps,-\frac{1}{s_0}\},  \;\ -\frac{1}{x}<s <s_0,  \;\ y \in \Omega\}.
\end{gather*}

As in the proof of Lemma \ref{lemma2}, we deduce that  for any $(x,y)$ with $x\leq \min\{\eps, -\frac{1}{s_0}\} $ and $y\in \Omega,$ $u(t,w)=0$ in a neighborhood of
$-(s_0+\frac{1}{x})< t<(s_0+\frac{1}{x}),$ and applying \eqref{tatres}, we conclude that
\begin{gather*}
u(t,z)=\p_t u(t,z)=0 \text{ provided } x<\eps, \;\ y\in \Omega \text{ and }  |t|+d_g(z,(x,y))< s_0 + \frac{1}{x}.
\end{gather*}
In particular, if $t=0,$ $f=\p_t u(0,z)=0$ if $d_g(z,(x,y))< s_0 + \frac{1}{x}.$ This concludes the proof of Theorem \ref{main}.

\end{document}